\documentclass[reqno, english,12pt]{amsart}
\RequirePackage{mathrsfs} \let\mathcal\mathscr
\usepackage{a4wide}
\usepackage[all]{xy}
\usepackage{enumerate}
\usepackage{color}
\usepackage{amssymb,amsmath,amsfonts,amsthm}

\usepackage{hyperref}
\usepackage{dsfont}
\usepackage{upgreek}
\usepackage{mathrsfs}
\usepackage{mathabx,yfonts}
\usepackage{enumerate}
\usepackage{enumitem}
\usepackage{graphicx}

\DeclareMathOperator{\Pic}{Pic} 
\DeclareMathOperator{\NS}{NS} 
 
\DeclareMathOperator{\rank}{rank}

\let\div\relax
\DeclareMathOperator{\div}{div}

\DeclareMathOperator{\Spec}{Spec}

\newtheorem{theorem}{Theorem}
\newtheorem{conjecture}[theorem]{Conjecture}
\newtheorem{lemma}[theorem]{Lemma}

\theoremstyle{definition}
\newtheorem{definition}[theorem]{Definition}
\newtheorem{example}[theorem]{Example}
\newtheorem{remark}[theorem]{Remark}

\numberwithin{theorem}{section}
\numberwithin{equation}{section}
\numberwithin{table}{section}

\DeclareSymbolFont{bbold}{U}{bbold}{m}{n}
\DeclareSymbolFontAlphabet{\mathbbold}{bbold}

\renewcommand{\epsilon}{\varepsilon}

\renewcommand{\leq}{\leqslant}
\renewcommand{\geq}{\geqslant}
\renewcommand{\#}{\sharp}
\renewcommand{\gg}{\ggg}

\renewcommand{\ll}{\lll}

\newcommand\vz{\mathbf{0}}

\newcommand\FF{\mathbb{F}}
\newcommand\PP{\mathbb{P}}
\newcommand\ZZ{\mathbb{Z}}

\newcommand\QQ{\mathbb{Q}}
\newcommand\RR{\mathbb{R}}

\newcommand\GG{\mathbb{G}}

\newcommand\Gm{\GG_\mathrm{m}}

\newcommand{\OO}{\mathcal{O}}

\newcommand{\archplaces}{{\Omega_\infty}}

\newcommand{\vx}{\mathbf{x}}

\newcommand{\vB}{\mathbf{B}}
\newcommand{\vy}{\mathbf{y}}

\newcommand{\norm}{\mathfrak{N}}

\newcommand{\abs}[1]{\left|#1\right|}
\newcommand{\absv}[1]{\left|#1\right|_v}

\newcommand\card{\#}

\newcommand{\locdegv}{{m_v}}
\newcommand\dg{m}

\begin{document}

\title
{
Rational points and non-anticanonical height functions
}

\author{Christopher Frei}
\address{
University of Manchester \\
School of Mathematics\\
Oxford Road\\
Manchester\\
M13 9PL\\
UK}
\email{christopher.frei@manchester.ac.uk}

\author{Daniel Loughran}
\address{
University of Manchester \\
School of Mathematics\\
Oxford Road\\
Manchester\\
M13 9PL\\
UK}
\email{daniel.loughran@manchester.ac.uk}

%\author{Efthymios Sofos}
%\address{
%Universiteit Leiden\\
%Mathematisch Instituut\\
%Niels Bohrweg 1\\
%Leiden\\
%2333 CA\\
%Netherlands
%}
%\email{e.sofos@math.leidenuniv.nl}

\begin{abstract}
A conjecture of Batyrev and Manin predicts the asymptotic behaviour of rational points of bounded
height on smooth projective varieties over number fields.
We prove some new cases of this conjecture for conic bundle
surfaces equipped with some non-anticanonical height functions. As a special case, we verify these conjectures for the first time
for some smooth cubic surfaces for height functions associated to certain ample line bundles.
\end{abstract}


\maketitle

\setcounter{tocdepth}{1}
\tableofcontents

\section{Introduction}\label{intro}
\subsection{The Batyrev--Manin conjecture}

This paper is concerned with counting rational points of bounded height on algebraic varieties.
Let $X$ be a smooth projective variety over a number field $k$ with $X(k) \neq \emptyset$ and let $D$ be a divisor on $X$.
Recall from the theory of heights that to each choice of adelic metric on the line bundle $\OO_X(D)$
one can associate a choice of height function $H$. (Note that this theory works equally well for $\QQ$-divisors, see e.g.~\cite[\S2]{CLT02}.)
If $D$ is \emph{big}, then such height functions
have the important property that the cardinality
$$N(U,H,B) = \#\{ x \in U(k) : H(x) \leq B\}$$
is finite for some open dense subset $U \subset X$ and all $B > 0$.
If $X$ is a Fano variety, or a variety which is close to being Fano,
then a conjecture of Batyrev and Manin \cite{BM90} predicts
an asymptotic formula of the shape
\begin{equation} \label{conj:BM}
	N(U,H,B) \sim c_{U,H} B^{a(D)} (\log B)^{b(D) -1},
\end{equation}
for some $c_{U,H} > 0$ and for certain exponents $a(D)$ and $b(D)$ defined in terms of the geometry of $D$ (we recall
the definitions of $a(D)$ and $b(D)$ in \S\ref{sec:conjectures}).

Much emphasis has been placed on the special case $D = -K_X$, i.e.~on anticanonical
height functions. Here $a(-K_X) = 1$ and $b(-K_X) = \rank \Pic X$.
In this case the asymptotic formula \eqref{conj:BM} has been verified in many special cases,
but is still open in general. For example for smooth cubic surfaces the best known upper bound
over $\QQ$ is $N(U,H,B) \ll_{\varepsilon} B^{4/3 + \varepsilon}$ due to Heath-Brown \cite{rogerbest},
under the assumption that $X$ contains $3$ coplanar lines and one takes $U$ to be the complement
of all the lines in $X$.

One of the observations in this paper is that one can sometimes
obtain better results for \emph{non-anticanonical heights}.
(Note that the harmonic analysis approach to counting rational points of bounded height \cite{CLT02,STBT07} usually works for all choices
of height function, rather than just anticanonical heights.)
If $X$ is Fano with $\Pic X = \ZZ$ then all heights are rational powers of anticanonical heights. 
So to obtain non-trivial non-anticanonical heights, one requires extra
geometric structure. In this paper we take
this to come from a \emph{conic bundle structure} (see \S\ref{sec:conics} for definitions).
%A case of particular interest is \emph{del Pezzo surfaces}, which are the Fano varieties
%of dimension $2$.

\subsection{Del Pezzo surfaces and conic bundle surfaces}
There are some ``easy" non-anticanonical heights which one can work with. 

\begin{example}
Let $X$ be a smooth cubic surface given as a blow-up $\pi:X \to \PP_k^2$  
in $6$ rational points in general position. Let $U \subset X$ be the complement
of the lines in $X$ and let $H$ be a height function associated to the divisor $\pi^*(-K_{\PP^2})$.
Then, due to the functoriality of heights, we have
$$N(U,H,B) = \# \{ x \in \pi(U) : H_{-K_{\PP^2}}(x) \leq B \},$$
which one can of course asymptotically estimate using Schanuel's theorem. 
However, the divisor $\pi^*(-K_{\PP^2})$ is not ample, which is reflected in the fact that
one is really just counting rational points on $\PP^2_k$ in this case.
\end{example}
Our first result concerns del Pezzo surfaces with a conic bundle
structure. Here we are able to deal with ample line bundles, so that the counting problem does not come from
a simpler variety.

\begin{theorem} \label{thm:dP}
	Let $X$ be a del Pezzo surface of degree $d$ over a number field $k$
        with a conic bundle structure $\pi:X \to \PP^1$. Let $U \subset X$ be the complement of the singular fibres of $\pi$ and assume that $U(k) \neq \emptyset$.
	Let $\alpha > 1$ if $d \geq 3$ and $\alpha>2$ if $d =2,1$. Let
	$H$ be a choice of height function associated to the $\QQ$-divisor $-K_X +
        \alpha F$,
where $F$ is the class of a fibre of $\pi$.
	Then 
	$$N(U,H,B) \sim c_{U,H} B, \quad \mbox{ as } B \to \infty,$$
	for some $c_{U,H} > 0$.
\end{theorem}

In Theorem \ref{thm:dP}, and throughout the rest of this paper, we take $\alpha$ to be a rational number.
For $\alpha \geq 0$, the $\QQ$-divisors $-K_X + \alpha F$, being the sum of an ample divisor and a semi-ample divisor, are ample.
Theorem \ref{thm:dP} agrees with the Batyrev--Manin conjecture (see \S \ref{sec:conjectures})
and applies, for example, to cubic surfaces with a line.
It proves, for the first time, a case of the Batyrev--Manin conjecture for smooth cubic surfaces
with respect to a height function associated to some ample line bundle. (Facts about del Pezzo surfaces
with a conic bundle structure can be found in \cite[\S5]{FLS16}.)

We also obtain results which apply to more general conic bundles.
Note that for a conic bundle surface $\pi:X \to \PP^1$, in general the anticanonical divisor $-K_X$ won't be big. However, if $F$
is a fibre of $\pi$, then the $\QQ$-divisors $-K_X + \alpha F$ will be big for sufficiently large $\alpha$, and these provide us with a natural class
of height functions satisfying the Northcott property on some open subset. Our result is as follows.

\begin{theorem} \label{thm:conic_bundle}
	Let $\pi:X \to \PP^1$ be a conic bundle surface with anticanonical divisor $-K_X$ and fibre $F$.
	Let $\alpha > (8-K_X^2)/3$ and let $H$ be a choice of height function associated to the $\QQ$-divisor $-K_X + \alpha F$.
	There exists a proper closed subset $E \subset X$ such that for all open dense subsets $U \subset X \setminus E$
	with $U(k) \neq \emptyset$ we have
	$$N(U,H,B) \sim c_{U,H} B, \quad \mbox{ as } B \to \infty,$$
	for some $c_{U,H} > 0$.
\end{theorem}

One needs to avoid a subset $E \subset X$ in Theorem \ref{thm:conic_bundle} as there can be accumulating subvarieties 
in general (the Northcott property may even fail on some curves for our height function). One can make very explicit which curves need
to be removed; see Theorem \ref{thm:HDP} and Example \ref{exa:Northcott}.

The leading constant $c_{U,H}$ in Theorem \ref{thm:conic_bundle} is the sum of the Peyre constants
of the smooth fibres of $\pi$ (see Theorem \ref{thm:main} for an explicit equation). We explain in \S\ref{sec:conjectures}
how this agrees with the conjectural constant proposed by Batyrev and Tschinkel in \cite{BT_Tamagawa}.
However, we show that Conjecture 3.5.1 from \emph{loc.~cit.}, concerning the distribution
of the Tamagawa measures in the family, is in fact \emph{false} in our case (the first counter-examples to this were found by Derenthal and Gagliardi \cite{DG16}).

\subsection{Higher dimensional conic bundles} \label{sec:HD}

Our results also apply to some other higher dimensional conic bundles over a number field $k$.
Our most general result in higher dimensions is Theorem \ref{thm:HDP};
for simplicity we state some special cases here.
For example, we can handle some hypersurfaces in biprojective spaces. %\chris{not all cases covered by this result are conic bundles}

\begin{theorem} \label{thm:bi_projective}
	Let $X \subset \PP^n \times \PP^2$ be a smooth biprojective hypersurface over $k$ of bidegree $(e,2)$
	for some $e$ and let $\pi: X \to \PP^n$ be the natural projection.
	Let $\OO_X(F) = \OO_X(1,0)$ and let $H$ a choice of height associated to 
	$-K_{X} + \alpha F$ for $\alpha > e$. Let $V \subset \PP^n$ be an open subset such that $\pi^{-1}(V) \to V$ is a smooth morphism and let  $U \subset \pi^{-1}(V)$ be an open subset with $U(k) \neq \emptyset$.
	Then 
	$$N(U,H,B) \sim c_{U,H} B, \quad \mbox{ as } B \to \infty,$$
	for some $c_{U,H} > 0$.
\end{theorem}
Theorem \ref{thm:bi_projective} applies to the family of diagonal plane
conics
\begin{equation} \label{eqn:diagonal}
	y_0x_0^2 + y_1x_1^2 + y_2x_2^2 = 0 \quad \subset \PP^2 \times \PP^2,
\end{equation}
and to the family of all plane conics
$$y_{00}x_0^2 + y_{01}x_0x_1 + y_{02}x_0x_2 + y_{11}x_1^2 + y_{12}x_1x_2 + y_{22}x_2^2 = 0 \quad \subset \PP^5 \times \PP^2.$$
Le Boudec \cite{Bou15} has proved upper and lower bounds of the correct order of magnitude for the anticanonical height function for \eqref{eqn:diagonal}, but an asymptotic formula for the anticanonical height is still unknown in this case.

Our methods also apply to cubic hypersurfaces.

\begin{theorem} \label{thm:cubic} 
  Let $X\subset\PP^{n+2}$ be a smooth cubic
  hypersurface of dimension $n+1$ over $k$ with a line $L \subset X$.  Let
  $\widetilde{X}$ be the blow-up of $L$ and $\pi: \widetilde{X} \to \PP^{n}$
  the morphism induced by projecting away from $L$. Let $F$ be the pull-back of the hyperplane class on $\PP^n$, 
  $\alpha > 2$ and $H$ a choice of height associated to
  $-K_{\widetilde{X}} + \alpha F$.  Let $V \subset \PP^n$ be an open subset such that $\pi^{-1}(V) \to V$ is a smooth morphism and let  $U \subset \pi^{-1}(V)$ be an open subset with $U(k) \neq \emptyset$. Then
	$$N(U,H,B) \sim c_{U,H} B, \quad \mbox{ as } B \to \infty,$$
	for some $c_{U,H} > 0$.
\end{theorem}

%\subsubsection{Cubic hypersurfaces with a line}
%Can we obtain interesting upper/lower bounds for rational points of bounded height on cubic hypersurfaces
%with a line, by blowing-up and using conic bundles? It would need to be able to beat the determinant method.
%For $X \subset \PP^n$ we would have $K_{\widetilde{X}} = f^*K_X + (n-3)E = $

%Lower bounds might be hard, but what about upper bounds?
%
%One line looks like it is not enough. What about 3 coplanar lines?

\subsection*{Acknowledgements} We thank Efthymios Sofos for suggesting the
problem and for useful discussions, and Per Salberger for pointing out some oversights in an earlier version of Theorem~\ref{thm:cubic}.
We are also grateful to Sho Tanimoto for help with the proof of Lemma \ref{lem:a_b}, and thank the referees for useful comments.

\section{Conic bundles and projective bundles} \label{sec:conics}

\subsection{Conic bundles}

\begin{definition}
	A conic bundle over a field $k$ is a proper morphism $\pi:X \to Y$
	 of smooth 
	varieties over  $k$ whose fibres are isomorphic to plane conics.
\end{definition}

The anticanonical bundle $\omega_{X}^{-1}$ induces the anticanonical bundle on each smooth
fibre of $\pi$. The pushforward $\mathscr{E} := \pi_* (\omega_{X}^{-1})$ is a vector
bundle of rank $3$ which induces an  embedding $X \hookrightarrow \PP(\mathscr{E})$ such that $\pi$
is compatible with the natural projection $\PP(\mathscr{E}) \to Y$. (This follows from an application
of \cite[Prop.~1.1.6, Lem.~1.1.8]{BenoistThesis}, for example.)
%Moreover $X$ is the zero locus of an element of $Q \in \HH^0(Y, \mathrm{S}^2\mathscr{E})$.
We follow the Grothendieck convention regarding projective bundles, namely that $\PP(\mathscr{E})$ denotes the space
of $1$-dimensional quotients of $\mathscr{E}$.

\subsection{Projective bundles}
We work with special choices of projective bundles in order to make the above set-up
and the resulting height functions explicit. We work over a field $k$, assumed to not have characteristic $2$
for simplicity. The theory presented here is just a mild generalisation to higher dimensions of \cite[\S2]{FLS16}.

Let $a_0,a_1,a_2 \in \ZZ$. We consider the following $\PP^2$-bundles over $\PP^n$:
\begin{equation} \label{def:F}
	\FF_n(a_0,a_1,a_2) := \PP_{\PP^n}( \OO_{\PP^n}(a_0) \oplus \OO_{\PP^n}(a_1) \oplus \OO_{\PP^n}(a_2)).
\end{equation}
Note that permuting the $a_i$ or replacing $(a_0,a_1,a_2)$ by $(a_0+f,a_1+f,a_2+f)$ gives an isomorphic $\PP^2$-bundle.
We let $M$ be the class of the relative hyperplane bundle and $F$ the pull-back of the hyperplane class on $\PP^n$
in the Picard group.
The bundle $\FF_n(a_0,a_1,a_2)$ can be constructed as an explicit quotient of the space $(\mathbb{A}^{n+1}\setminus 0) \times (\mathbb{A}^{3}\setminus 0)$
by the following action of $\Gm^2$:
\begin{equation} \label{eqn:Gm^2}
	(\lambda,\mu) \cdot (y_0,\dots,y_n;x_0,x_1,x_2) = (\lambda y_0, \dots, \lambda  y_n;\lambda^{-a_0} \mu x_0,\lambda^{-a_1} \mu x_1,\lambda^{-a_2} \mu x_2).
\end{equation}
We therefore obtain well-defined coordinates $(y_0:\dots:y_n;x_0:x_1:x_2) = (y;x)$ on $\FF_n(a_0,a_1,a_2)$
which are bihomogenous with respect to the action \eqref{eqn:Gm^2}.

A hypersurface of bidegree $(2,e)$ in $\FF_n(a_0,a_1,a_2)$
has an equation of the shape
\begin{equation} \label{eqn:conic_bundle}
	\sum_{0 \leq i,j \leq 2} f_{i,j}(y)x_ix_j= 0.
\end{equation}
Throughout this paper, we follow the convention that $f_{i,j} = f_{j,i}$. We work with the natural projection $\pi:X \to \PP^n$, which restricts to a conic bundle morphism above any open subset $V \subset \PP^n$ for which $\pi^{-1}(V) \to V$ is flat. 
Bihomogeneity implies that the degrees of the $f_{i,j}$ are given by the following matrix
\begin{equation} \label{eqn:degree_matrix}
\left( \begin{array}{ccc}
2a_0 + e & a_0 + a_1 + e & a_0 + a_2 + e \\
a_0 + a_1 + e & 2a_1 + e & a_1 + a_2 + e \\
a_0 + a_2 + e & a_1 + a_2 + e & 2a_2 + e \end{array} \right).
\end{equation}

\begin{lemma} \label{lem:P}
	Let $X$ be a smooth hypersurface of bidegree $(2,e)$ in 
	$\FF_n(a_0,a_1,a_2)$. Let $\Delta \in k[y_0,\ldots,y_n]$ be the associated
	discriminant polynomial, i.e.~the determinant of the matrix of the 
	quadratic form defining \eqref{eqn:conic_bundle}. Then
	\begin{enumerate}
		\item $-K_X = M + ((n+1) - a_0 - a_1 - a_2 - e)F$.
		\item The discriminant $\Delta$ is a squarefree homogeneous polynomial with $\deg \Delta = 2(a_0 + a_1 + a_2) + 3e$.
%		\item If $k$ is a number field, then a choice of height function associated to $-K_X + \alpha F$ is
	\end{enumerate}
      \end{lemma}
\begin{proof}
	By \cite[Ex.~III.8.4]{Har77}, the canonical divisor of $\FF_n(a_0,a_1,a_2)$ is $( -(n+1) + a_0 + a_1 + a_2)F -3M$.
	Part (1) therefore follows from the adjunction formula.
	
	For Part (2), the degree of $\Delta$ is calculated by noting that it is simply the trace of the matrix \eqref{eqn:degree_matrix}.
	To prove that $\Delta$ is squarefree, we may work locally around each divisor in $\PP^n$. Let $R$ be the local ring at some codimension $1$
	point of $\PP^n$ and consider $X_R \to \Spec R$; this is regular as $X$ is regular.
	As $R$ is a discrete valuation ring, the vector bundle $\OO_{\PP^n}(a_0) \oplus \OO_{\PP^n}(a_1) \oplus \OO_{\PP^n}(a_2)$ trivialises over $R$.
	Moreover, as $R$ is a local ring with $2 \in R^*$,
	we may diagonalise the equation of $X_R$ over $R$ \cite[Cor.~I.3.4]{MH73} to find that
	$$X_R: \quad r_0x_0^2 + r_1 x_1^2 + r_2x_2^2 = 0 \quad \subset \PP^2_R,$$
	where $r_i \in R$.
	In particular the base change $\Delta_R \in R$ of $\Delta$ is given by 
	$\Delta_R = ur_0r_1r_2$ for some $u \in R^*$. However, as $X_R$ is regular, a calculation
	shows that the valuation of $r_0r_1r_2$ is at most $1$. Applying this to each codimension $1$ point proves the claim.
\end{proof}

\section{Proof of results}

\subsection{Statement}
We begin by considering higher dimensional hypersurfaces in the $\PP^2$-bundles over $\PP^n$ from \eqref{def:F}.  The results from \S \ref{sec:HD}
will be proved using the following.

\begin{theorem} \label{thm:HDP} 
  Let $a_0\leq a_1\leq a_2 \in \ZZ$ and let $X \subset
  \FF_n(a_0,a_1,a_2)$ be a smooth hypersurface of bidegree $(e,2)$ over a
  number field $k$, for some $e \in \ZZ$. Let $\pi: X \to \PP^n$
  be the natural projection and $F$ the pull-back of the hyperplane class on $\PP^n$. Let $V \subset \PP^n$ be an open subset such that $\pi^{-1}(V) \to V$ is a smooth morphism and let $U \subset \pi^{-1}(V)$ be an open subset with $U(k) \neq \emptyset$ which  does not meet the hypersurface $x_2=0$. 
  Let $H$ be a choice of height associated to
  $-K_{X} + \alpha F$ for some $\alpha > e+2(a_0+a_1+a_2)/3$.
  Then
  $$N(U,H,B) \sim c_{U,H} B, \quad \mbox{ as } B \to \infty,$$
  for some $c_{U,H} > 0$.
\end{theorem}
For example, one can take $V$ to be the complement of the discriminant locus $\Delta(y) = 0$ in $\PP^n$.
We choose our $U$ as in the statement of Theorem \ref{thm:HDP} as the fibres over the discriminant locus and the hyperplane $x_2 = 0$ are
accumulating
subvarieties in general. (The hyperplane $x_2 = 0$ defines a degree $2$ multisection of $\pi$.)
That one needs to remove $x_2 = 0$
is illustrated by the following example (note that $x_2$ is special only because we stipulated that $a_2 \geq \max\{a_1,a_0\}.$)

\begin{example} \label{exa:Northcott}
	Let $a \in \ZZ$ satisfy $3 \mid a$ and $a > 9$. Take $\alpha = 2a/3 + 1$ and 
	$D = -K_X + \alpha F$. Let $f$ be a squarefree 
	binary form of degree $2a$ and consider the smooth surface
	$$X: \quad  x_0^2 - x_1^2 = f(s,t) x_2^2 \quad \subset \FF_1(0,0,a),$$
	equipped with its natural conic bundle structure $\pi: X \to \PP^1$.
	Note that $\alpha$ satisfies the assumptions of Theorem \ref{thm:HDP}, with $e=0$.
	
	However, let $C$ be the curve given by $x_0 + x_1 = x_2 = 0$
	(this is a section of $\pi$).	
	Lemma \ref{lem:P} shows that $D =  -K_X + \alpha F =M + (3 - a/3)F$.
	However $C \cdot F = 1$ and $C \cdot M = 0$, hence
	$C \cdot D  = 3 - a/3 < 0$. Thus $C$ contains infinitely many rational points
	of height less than any given $B > 0$ with respect to $D$
	(the failure of the Northcott property on $C$ can also be verified 
	using the explicit description of the height function given in Lemma \ref{lem:standard_height}).
	Thus $C$ must be removed for the conclusion of Theorem \ref{thm:HDP} to hold.
\end{example}
  
  %CF: should we move this remark to directly after the theorem? Also, maybe tis
%formulations need a slight update. 
%DL: Moved! What update are you referring to?
%CF: The $\PP^2$ could be a $\PP^n$, and in "one can make the height
%functions...explicit", one could maybe mention Section 2? 
%DL: How about now?
\begin{remark}
For conic bundles inside the special $\PP^2$-bundles $\PP(\mathcal{E})$ 
with $\mathcal{E}$ a direct sum of three line bundles, one can make 
the height functions and equations of the conic bundle explicit (see \S\ref{sec:Heights} and \S\ref{sec:fibration}). 
From a highbrow perspective, this is because $\PP(\mathcal{E})$ is \emph{toric} in this case,
as reflected in its description \eqref{eqn:Gm^2} as a quotient of an open in an affine
space.

To generalise our method  to general rank $3$ vector bundles $\mathcal{E}$ on $\PP^n$,
one requires an explicit
description of the Cox ring of $\PP(\mathcal{E})$. 
However, it does not even seem to be known whether this ring is always finitely generated
when $\rank \mathcal{E} = 3$
(i.e.~whether $\PP(\mathcal{E})$ is a Mori dream space).
Special cases where finite generation is known include the tangent bundle of $\PP^n$ \cite[Thm.~5.9]{HS10}.
It would be interesting to study conic bundles inside other projective bundles.
(Note that subtleties only arise when $n > 1$,
as every vector bundle on $\PP^1$ is a direct sum of lines bundles \cite[Ex.~V.2.6]{Har77}).
\end{remark}

%Here the height function takes the following form
%$$H(y_0:\cdots:y_n; x_0:x_1:x_2) = \prod_v \max\{|y_i|_v\} \max\{|? x_i\},$$
%where $?x_i$ runs over all monomials of bidegree ? in $y_i$ and $x_i$.

%Is actually not quite the main theorem.... Need to do more work for conic
%bundle surfaces.

%We use the following notation from \cite{FLS16}. Let $(a_0,a_1,a_2) \in \ZZ^2$. We let 
%$$\FF(a_0,a_1,a_2) = \PP( \OO_{\PP^1}(a_1) \oplus \OO_{\PP^1}(a_2) \oplus \OO_{\PP^1}(a_3))$$
%be the projectivisation of the rank $3$ vector bundle $\OO_{\PP^1}(a_1) \oplus \OO_{\PP^1}(a_2) \oplus \OO_{\PP^1}(a_3)$
%on $\PP^1$ (we use \cite{Rei93} as our reference for projective bundles).

%Let us now prove Theorem \ref{thm:conic_bundle_eqn}. We have $\deg \Delta(s,t) = 2(a_0 + a_1 + a_2) + 3e$.

\subsection{Proof strategy of Theorem \ref{thm:HDP}}
Each smooth conic $Q$ in the family has $c_Q B +
o_Q(B)$ points of height at most $B$ for some $c_Q \geq 0$. We will show that
the sum over these contributions is convergent via the dominated
convergence theorem. To achieve this we require a uniform upper bound for the
number of rational points on each conic; this is provided by the following
result due to Browning and Swarbrick Jones. 

We fix, once and for all, a set $\mathcal{C}$ of integral representatives of the ideal classes of $k$. All
implied constants in the paper are allowed to depend on $\mathcal{C}$. We let
$\tau_k(\cdot)$ be the divisor function on ideals of $\OO_k$ and  $\archplaces$
be the set of archimedean places of $k$.

\begin{lemma}[{\cite[Theorem 2.3]{mikey}}] \label{lem:BSJ} Let $A \in
  \mathrm{M}_3(\OO_k)$ be a symmetric matrix which is invertible over $k$ and
  let $Q$ be the associated ternary quadratic form over $\OO_k$. Let
  $\vB_1,\vB_2,\vB_3 \in [0,\infty)^\archplaces$, with
  $\vB_i=(B_{i,v})_{v\in\archplaces}$. Let $N(A,\vB_1,\vB_2,\vB_3)$ be the number
  of all $x\in\PP^2(k)$ that have a representative
  $\vx=(x_0,x_1,x_2)\in\OO_k^3$ with $Q(\vx)=0$, 
        \begin{equation*}
          x_0\OO_k+x_1\OO_k+x_1\OO_k\in\mathcal{C},\ \text{ and }\ \absv{x_i}\leq
        B_{i,v} \text{ for all } 1\leq i\leq 3 \text{ and all } v\in\archplaces. 
        \end{equation*}

  Write $\Delta(A):=\det A$ and let $\Delta_0(A)$ be the ideal of $\OO_k$
  generated by the $(2\times 2)$-minors of $A$. Then
  \begin{equation*}
    N(A,\vB_1,\vB_2,\vB_3) \ll
    \tau_k(\Delta(A))\left(\left(\frac{\norm(\Delta_0(A))^{3/2}\prod_{v\in\archplaces}(B_{1,v}B_{2,v}B_{3,v})^\locdegv
            }{\abs{N_{k/\QQ}(\Delta(A))}}\right)^{1/3} + 1\right)
  \end{equation*}
\end{lemma}
Here $\norm$ denotes the ideal norm and $N_{k/\QQ}:k \to \QQ$ the field norm.
Observe  that the implied constant in the above result does not
depend on the matrix $A$.

\subsection{Heights} \label{sec:Heights}
To implement our strategy, we need to make the height functions explicit.
Denote by $(\vy,\vx)$ the coordinates in $\mathbb{A}^{n+1} \times \mathbb{A}^3$
and by $(y;x) = (y_0:\dots:y_n; x_0:x_1:x_2)$ the corresponding coordinates
on $\FF_n(a_0,a_1,a_2)$.
Note that we prove Theorem \ref{thm:HDP} for completely general choices of
height function $H$ associated to $-K_X + \alpha F$. We do this by relating $H$ to a
``standard'' choice of height $H^*(y;x)$.

We let $m = [k:\QQ]$. For a place $v$ of $k$ we let $\locdegv = [k_v : \QQ_w]$ where $w$ is
the unique place of $\QQ$ below $v$. We similarly let $| \cdot |_v$ be the absolute value
on $k_v$ extending the
%CF
 standard
 absolute value $|\cdot|_w$ on $\QQ_w$.
%DL: Is this really the ``standard'' absolute value? People use different normalisations all the time. I would prefer to make this explicit as we had before.
%CF: I just added the word "standard" to the absolute values on $\QQ$, nothing
%was explicit before. And I think
%the real and p-adic absolute values on $\QQ_p$ can be called standard. How
%else would one define them? 
%DL: Sorry I had misread this. Looks fine!

\begin{lemma}\label{lem:standard_height}
	Let $A = n + 1 + \alpha - (a_0 + a_1 + a_2 + e)$ and define
	$H^*(y;x)=\prod_v H^*_v(y,x)^{\locdegv}$ with local factors
	\begin{equation} \label{def:Height}
	  H^*_v(y,x) = \max_{i}\{|y_i|_v\}^{A}\max_j\{ \max_i\{|y_i|_v\}^{a_j}|x_j|_v\}.
	\end{equation}
	Then $H^*$ is a height on $X$ associated to $-K_X + \alpha F$. In particular, there exist $c_1, c_2 > 0$ such that
	for all $(y;x) \in X(k)$ we have
	\begin{equation}
		c_1 H^*(y;x) \leq H(y;x) \leq c_2 H^*(y;x).
	\end{equation}
\end{lemma}
\begin{proof}
	The construction of the height follows from Weil's height machine as follows.
	Let $M$ be the relative hyperplane class. The line bundle $\OO_X(M)$ is generated by
	the global sections $y_i^{a_j}x_j$. In particular, a choice of height
	function is given by
	$$H_M(y,x) = \prod_v \max_j\{ \max_i\{|y_i|_v\}^{a_j}|x_j|_v\}^\locdegv.$$
	Similarly, the bundle $\OO_X(F)$ is generated by the global sections $y_i$. 
	By Lemma \ref{lem:P} we know that $-K_X = M + ((n+1) - a_0 - a_1 - a_2 - e)F$, 
	thus $-K_X + \alpha F = M + AF$ in $\Pic X$.
	Hence a choice of height is given by $H_M \cdot H_F^A$, which is exactly the height $H^*$.
	The second part of the lemma follows from standard properties of heights. 
\end{proof}

For $y \in \PP^n(k)$ we denote by $H(y)$ its usual $\OO(1)$-height, and for $\vy\in\OO_k^{n+1}$ we write
$H_\infty(\vy):=\prod_{v\mid\infty}\max_i\{\absv{y_i}\}^{m_v}$. 

\subsection{Fibration} \label{sec:fibration}
A hypersurface $X$ of bidegree $(e,2)$ in $\FF(a_0,a_1,a_2)$ is cut out by a bihomogeneous form
\begin{equation*}
  Q_{\vy}(\vx) = Q(\vy,\vx) = \sum_{0\leq i,j\leq 2}f_{i,j}(\vy)x_ix_j,
\end{equation*}
with $f_{i,j}(\vy)\in\OO_k[\vy]$ a form in $n+1$ variables of degree $\deg f_{i,j}=a_i+a_j+e$ over $\OO_k$. We 
assume that $f_{i,j}=f_{j,i}$ for all $i,j$. Recall that $\pi : X \to \PP^n$ is the natural projection $(y;x)\mapsto y$. Let 
$\Delta(\vy)=\det(f_{i,j}(\vy))_{i,j}$ be the discriminant of $\pi$, a form in
$n+1$ variables over $\OO_k$. For every $\vy\in\OO_k^{n+1}$, we let
$\Delta_0(\vy) \subseteq \OO_k$ 
be the ideal generated by all $(2\times 2)$-minors of the matrix $(f_{i,j}(\vy))_{i,j}$. 

Let $\vy \in k^{n+1}\smallsetminus\{\mathbf{0}\}$ with $\Delta(\vy) \neq 0$. Then the plane conic $C_{\vy}$ defined in $\PP^2$ by the ternary quadratic
form $Q_{\vy}$ is isomorphic to the fibre $X_y$ above $y\in\PP^n(k)$ 
via $\vx\mapsto (\vy;\vx)$. The restriction of $H$ to the smooth fibre $X_y$ is an anticanonical
height on $X_y$, which pulls back to an anticanonical height
$H_{y}(\vx)=H(\vy;\vx)$ on $C_{\vy}$.

\begin{lemma}\label{lem:minors_norm_bounded_general}
For any $\vy \in\OO_k^{n+1}$ with $\Delta(\vy) \neq 0$ we have $\norm \Delta_0(\vy)^3 \ll \abs{N_{k/\QQ}( \Delta(\vy))}^2$.
\end{lemma}

\begin{proof}
  Denote the $(2\times 2)$-minors of $M(\vy):=(f_{i,j}(\vy))_{i,j}$
  by $m_{i,j}(\vy)$, for $0\leq i,j\leq 2$. Let $M^*(\vy):=
  ((-1)^{i+j}m_{j,i}(\vy))_{i,j}$ be the adjugate matrix, then $\det M^*(\vy)\in\Delta_0(\vy)^3$. From Cramer's rule, we get $M(\vy)M^*(\vy) = \Delta(\vy)I$, where $I$ is the $(3\times 3)$-identity matrix. Taking determinants, we end
  up with $\det M^*(\vy) = \Delta(\vy)^2$. Thus, as ideals of $\OO_k$ we have
  $\Delta_0(\vy)^3\mid \Delta(\vy)^2$, and the lemma follows upon taking norms.
\end{proof}

%chris: old version with polynomials
% \begin{lemma}\label{lem:minors_norm_bounded_general}
% 	As ideals of $\OO_k[\vy]$ we have $ \Delta_0(\vy)^3 \mid \Delta(\vy)^2$. 
% 	In particular, for all $\vy \in\OO_k^{n+1}$ with $y_0\OO_k+ \dots + y_n\OO_k\in\mathcal{C}$
% 	we have $\norm \Delta_0(\vy)^3 / \norm \Delta(\vy)^2 \ll 1$.
% \end{lemma}

% \begin{proof}
%   We denote the $(2\times 2)$-minors of $M(\vy):=(f_{i,j}(\vy))_{i,j}$
%   by $m_{i,j}(\vy)$, for $0\leq i,j\leq 2$. Let $M^*(\vy):=
%   ((-1)^{i+j}m_{j,i}(\vy))_{i,j}$ be the adjugate matrix and let $f(\vy) \in \Delta_0(\vy)$.
%   From Cramer's rule, we get
%   \begin{equation*}
%     M(\vy)M^*(\vy) = \Delta(\vy)I,
%   \end{equation*}
%   where $I$ is the $(3\times 3)$ identity matrix. Taking determinants, we end
%   up with
%   \begin{equation*}
%     \det M^*(\vy) = \Delta(\vy)^2.
%   \end{equation*}
%   However $f(\vy)^3\mid\det M^*(\vy)$, and thus $f(\vy)^2\mid
%   \Delta(\vy)$, as required.
% \end{proof}

\begin{lemma}\label{lem:conics_bound}
  Let $\vy\in\OO_k^{n+1}$ with $\Delta(\vy) \neq 0$  and $y_0\OO_k+ \dots + y_n\OO_k\in\mathcal{C}$. Then
  \begin{equation*}
    N(C_{\vy},H_{\vy},B) \ll \tau_k(\Delta(\vy))\left(\frac{B \norm(\Delta_0(\vy))^{1/2}}{H_\infty(\vy)^{(a_0+a_1+a_2)/3+A}\abs{N_{k/\QQ}(\Delta(\vy))}^{1/3}}+1\right). 
  \end{equation*}
\end{lemma}

\begin{proof}
  Write $H_\infty^*(\vy,\vx) :=
  \prod_{v\mid\infty}H_v^*(\vy,\vx)^\locdegv$. Every rational point $x\in C_{\vy}(k)$ has a representative $\vx = (x_0,x_1,x_2)\in\OO_k^3$ with
  $Q_{\vy}(\vx)=0$, which satisfies moreover 
  \begin{align}
    &x_0\OO_k+x_1\OO_k+x_2\OO_k\in\mathcal{C},\quad and\label{eq:conic_rep_coprime}\\
    &H_v^*(\vy,\vx)\asymp
    H_\infty^*(\vy,\vx)^{1/\dg} \text{ for all }v\mid\infty.\label{eq:conic_rep_dirichlet}
  \end{align}
  To obtain \eqref{eq:conic_rep_dirichlet}, we used Dirichlet's unit theorem (cf.~\cite[\S13.4]{Ser97}). If $H(y;x)\leq B$, then the representative $\vx$ satisfies
  % Thus, it suffices to bound the number of such representatives $\vx$ that
  % satisfy the height bound
  % $H(y;x)\leq B$.
  % For these representatives, we have
  \begin{equation*}
    H_\infty^*(\vy,\vx)\ll H^*(y;x)\ll H(y;x)\leq B,
  \end{equation*}
  where the first estimate holds due to \eqref{eq:conic_rep_coprime} and the
  analogous assumption for $\vy$, and the second estimate is due to
  Lemma \ref{lem:standard_height}. With \eqref{eq:conic_rep_dirichlet}, this yields
  \begin{equation*}
    \absv{x_j}\leq\frac{H_v^*(\vy,\vx)}{\max_i\{\absv{y_i}\}^{A+a_j}} \ll \frac{B^{1/\dg}}{\max_i\{\absv{y_i}\}^{A+a_j}}, \quad j = 0,1,2.
  \end{equation*}
  With this observation, the desired bound now follows  from
  Lemma~\ref{lem:BSJ}.
\end{proof}

\begin{lemma} \label{lem:y_small}
	Let $(y;x) \in U(k)$ satisfy $H(y;x) \leq B$. Then
    $
	  H(y) \ll B^{1/(A+a_2)}.
    $
\end{lemma}
\begin{proof}
	Since $x_2\neq 0$ on $U$, this follows immediately from Lemma \ref{lem:standard_height}.
\end{proof}

\begin{lemma}\label{lem:y_height_bound}
  Let $\vy\in \OO_k$ with $\Delta(\vy) \neq 0$ and $y_0\OO_k+ \dots + y_n\OO_k \in\mathcal{C}$. Let $\epsilon>0$. Then
  \begin{equation*}
    \frac{N(X_{y}\cap U,H, B)}{B}\ll_\epsilon
    \frac{1}{H(y)^{(a_0+a_1+a_2)/3+A-\epsilon}}
    + \frac{1}{H(y)^{A+a_2-\epsilon}}. 
  \end{equation*}
\end{lemma}

\begin{proof}
  The estimate clearly holds when $N(X_{y}\cap U, H, B)=0$. Thus, let us assume that $(y;x)\in
  U(k)$ with $H(y;x)\leq B$. Then $H(y)\asymp H_\infty(y)$ and $B\gg H(y)^{A+a_2}$ by Lemma
  \ref{lem:y_small}. The lemma
  is now an immediate consequence of the bounds in Lemma
  \ref{lem:minors_norm_bounded_general} and Lemma \ref{lem:conics_bound},
  the isomorphism $C_{\vy}\cong X_{y}$, and the fact that, for all $\epsilon>0$,
  \begin{equation*}
    \tau_k(\Delta(\vy))\ll_\epsilon |N_{k/\QQ}(\Delta(\vy))|^{\epsilon}
    =\prod_{v\mid\infty}\absv{\Delta(\vy)}^{\epsilon\locdegv}\ll_\epsilon H_\infty(\vy)^{\epsilon\deg\Delta}. \qedhere
  \end{equation*}
\end{proof}

\begin{lemma}\label{lem:upper_bound_conv}
  There exists $\epsilon > 0$ such that the infinite series
  \begin{equation*}
    \sum_{y\in\PP^n(k)}\frac{1}{H(y)^{A+(a_0+a_1+a_2)/3-\epsilon}} \quad and \quad 
    \sum_{y\in\PP^n(k)}\frac{1}{H(y)^{A+a_2 - \varepsilon}}
  \end{equation*}
  are convergent.
\end{lemma}

\begin{proof}
	Recalling the definition of $A$ given in Lemma \ref{lem:standard_height} and our assumptions
	on $\alpha$ in Theorem \ref{thm:HDP}, we have 
	\begin{align*}
		A + a_2 \geq A + (a_0+a_1+a_2)/3  = n+1 + \alpha - e - 2(a_0+a_1 +a_2)/3 &> n+1. 
	\end{align*}
  Choosing $\varepsilon > 0$ sufficiently small, we therefore arrange
  the exponents of $H(y)$ to be strictly larger than $n+1$. Hence
  the result follows from Schanuel's  theorem \cite{Sch79} and partial summation.
\end{proof}

Theorem \ref{thm:HDP} is an immediate consequence of the following
result, which gives more precise information about the leading constant $c_{U,H}$.
\begin{theorem} \label{thm:main}\
  Under the same assumptions of Theorem \ref{thm:HDP}, the following hold.
  \begin{enumerate}
  \item For $y\in\PP^n(k)$ with $\Delta(y) \neq 0$, we have
    \begin{equation*}
      N(X_{y}\cap U,H,B) = c_{y}B(1+o_{y}(1))\text{, as }B\to\infty,
    \end{equation*}
    where $c_{y}\geq 0$ is Peyre's constant for $X_{y}$ (or
    $c_{y}=0$ if $X_{y}\cap U=\emptyset$).
  \item The sum
    \begin{equation*}
      c_{U,H}:= \sum_{\substack{y\in\PP^n(k)\\ \Delta(y) \neq 0}}c_{y}
    \end{equation*}
    is convergent.
  \item We have
    \begin{equation*}
      N(U,H,B) = c_{U,H}B(1+o(1))\text{, as }B\to \infty.
    \end{equation*}
  \end{enumerate}
\end{theorem}

\begin{proof}
If $X_{y}\cap U$ is not empty, then it differs from $X_{y}$
in only finitely many points. Since the restriction of $H$ to $X_{y}$ is an
anticanonical height, the asymptotical formula in \emph{(1)} is just Manin's conjecture for
$\PP^1$, proved in \cite[Cor.~6.2.18]{peyre}.

For \emph{(2)} and \emph{(3)}, we choose representatives $\vy$ for the points
$y\in\PP^n(k)$ as in Lemma~\ref{lem:y_height_bound}. Then
\begin{equation*}
  \sum_{\substack{y\in\PP^n(k)\\ \Delta(y) \neq 0}}c_{y} =
  \sum_{\substack{y\in\PP^n(k)\\ \Delta(y) \neq 0}}\lim_{B\to\infty}\frac{N(X_{y}\cap U,H,B)}{B}.
\end{equation*}
  The upper bound in Lemma \ref{lem:y_height_bound} is independent of
  $B$, and summable by Lemma~\ref{lem:upper_bound_conv}. Thus, the dominated convergence theorem yields
  \emph{(2)} and moreover allows us to exchange sum and limit, giving \emph{(3)}.
\end{proof}

This completes the proof of Theorem \ref{thm:HDP}. \qed

\subsection{Conic bundle surfaces}
In the case of conic bundle surfaces we can obtain a slightly stronger result than Theorem \ref{thm:HDP}.

\begin{theorem} \label{thm:conic_bundle_eqn}
  Let $a_0\leq a_1\leq a_2 \in \ZZ$ and let $X \subset
  \FF_1(a_0,a_1,a_2)$ be a smooth hypersurface of bidegree $(e,2)$ over a
  number field $k$, for some $e \in\ZZ$. Let $\pi: X \to \PP^1$
  be the natural projection and $F$ the pull-back of the hyperplane class on $\PP^1$.
  Let $U \subset X$ be an open dense
  subset with $U(k) \neq 0$ which does not meet any singular fibre of $\pi$ and which
  does not meet the hypersurface $x_2=0$. Let $H$ be a choice of height associated to
  $-K_{X} + \alpha F$ for some $\alpha > a_0 + a_1 + e$.  Then
  $$N(U,H,B) \sim c_{U,H} B, \quad \mbox{ as } B \to \infty,$$
  for some $c_{U,H} > 0$.
\end{theorem}

\subsubsection{Proof of Theorem \ref{thm:conic_bundle_eqn}}
The proof is a minor variant of the proof of Theorem~\ref{thm:HDP}. We achieve this by
performing a more careful analysis of the factors $\Delta_0(\vy)$ and $\Delta(\vy)$.
We keep the notation from the proof of Theorem~\ref{thm:HDP}.
The following is well-known,
and follows from a minor variant of the proof of \cite[Lemma 7]{broberg_cubic}.

\begin{lemma}\label{lem:minors_norm_bounded}
  Let $(y_0,y_1)\in\OO_k^{2}$ with $y_0\OO_k+y_1\OO_k\in\mathcal{C}$. Then $\norm
  \Delta_0(y_0,y_1)\ll 1$.
\end{lemma}

Note that Lemma \ref{lem:minors_norm_bounded} is specific to the case $n=1$; 
the bound $\norm  \Delta_0(y_0,\dots,y_n)\ll 1$ need not hold in general if $n > 1$.

\begin{lemma}\label{lem:s_t_height_bound}
  Let $y_0,y_1\in \OO_k$ with $y_0\OO_k+y_1\OO_k\in\mathcal{C}$, such that the fibre $X_{y}$
  is smooth. Let $\epsilon>0$. Then
  \begin{equation*}
    \frac{N(X_{y}\cap U,H, B)}{B}\ll_\epsilon\left(\frac{1}{H(y)^{(a_0 + a_1+a_2)/3+A-\epsilon}\abs{N_{k/\QQ}(\Delta(\vy))}^{1/3}}
    + \frac{1}{H(y)^{A+a_2-\epsilon}}\right). 
  \end{equation*}
\end{lemma}

\begin{proof}
	In light of Lemma \ref{lem:minors_norm_bounded}, the proof is a minor modification of the proof of Lemma~\ref{lem:y_height_bound}.
%  The estimate clearly holds when $N_{H}(X_{(y_0:y_1)}\cap U, B)=0$. Thus, let us assume that $\vx=(s:t;x_0:x_1:x_2)\in
%  U(k)$ with $H(\vx)\leq B$. Then $H(y_0:y_1)\asymp H_\infty(y_0,y_1)$ and $B\gg H(y_0:y_1)^{A+a_2}$ by Lemma
%  \ref{lem:s_t_height_bound}. The lemma
%  is now an immediate consequence of the bound in Lemma \ref{lem:conics_bound},
%  the isomorphism $C_{(y_0,y_1)}\cong X_{(y_0:y_1)}$, and the fact that, for all $\epsilon>0$,
%  \begin{equation*}
%    \tau_k(\Delta(y_0,y_1))\ll_\epsilon N_{k/\QQ}(\Delta(y_0,y_1))^{\epsilon}=\prod_{v\mid\infty}\absv{\Delta(y_0,y_1)}^{\epsilon\locdegv}\ll_\epsilon H_%\infty(y_0,y_1)^{\epsilon\deg\Delta}. \qedhere
%  \end{equation*}
\end{proof}

\begin{lemma}\label{lem:upper_bound_2_conv}
  For each $y\in\PP^1(k)$, we choose a fixed representative $\vy=(y_0,y_1)\in\OO_k^2$
  with $y_0\OO_k+y_1\OO_k\in\mathcal{C}$. Then, for small enough $\epsilon$, the infinite series
  \begin{equation*}
    \sum_{\substack{y\in\PP^1(k)\\\Delta(\mathbf{y})\neq 0}}\frac{1}{H(y)^{(a_0 + a_1+a_2)/3+A-\epsilon}\abs{N_{k/\QQ}(\Delta(\mathbf{y}))}^{1/3}}
    \quad and \quad 
    \sum_{y\in\PP^1(k)}\frac{1}{H(y)^{A+a_2 - \varepsilon}}
  \end{equation*}
  converge.
\end{lemma}

\begin{proof}
  The second series converges as $A + a_2 = 2 + \alpha - (a_0 + a_1 + e) > 2$,
  by the choice of $\alpha$.
  
  As for the first series, the summand is invariant under multiplication of $(y_0,y_1)$ by elements
  of $\OO_k^\times$. % As a consequence of Dirichlet's unit theorem (see ,
  %DL: changed this
  By a standard argument using Dirichlet's unit theorem
  (see e.g.~\cite[\S13.4]{Ser97}), we may thus
  assume that the fixed representative $(y_0,y_1)$ for each point in $\PP^1(k)$ satisfies
  \begin{equation*}
    \max\{\absv{y_0},\absv{y_1}\}^{\locdegv}\asymp
    H_\infty(\vy)^{1/|\archplaces|}\asymp H(y)^{1/|\archplaces|}\text{ for
      all }v\mid\infty.
  \end{equation*}
  Then it follows from \cite[Theorem 2.4]{mikey} and a dyadic splitting of the sum that
  \begin{equation*}
    \sum_{\substack{y\in\PP^1(k)\\H(y)\leq B\\\Delta(\vy)\neq
        0}}\frac{1}{\abs{N_{k/\QQ}(\Delta(\vy))}^{1/3}}\ll_{\epsilon} B^{2-\frac{\deg\Delta}{3}+\epsilon}.
  \end{equation*}
  The lemma follows from this by partial summation and Lemma \ref{lem:P}, using the observation that
  \begin{align*}
    \frac{(a_0+a_1+a_2)}{3}+A+\frac{\deg\Delta}{3} & = A + a_0 + a_1+a_2+e>2,
  \end{align*}
  since $A+a_2>2$ and $a_0 + a_1 + e = \deg f_{0,1} \geq 0$
  (cf.~\eqref{eqn:degree_matrix}).
\end{proof}

Thus, replacing Lemma \ref{lem:y_height_bound} and Lemma \ref{lem:upper_bound_conv} in the proof of Theorem \ref{thm:main} by Lemma \ref{lem:s_t_height_bound} and Lemma \ref{lem:upper_bound_2_conv}, respectively, we see that the conclusions of Theorem \ref{thm:main} remain valid in case $n=1$  under the weaker assumptions of Theorem \ref{thm:conic_bundle_eqn}. This concludes our proof of Theorem \ref{thm:conic_bundle_eqn}. \qed

\subsection{Proof of Theorem \ref{thm:conic_bundle}}
Let $X$ be as in Theorem \ref{thm:conic_bundle}. As every vector bundle on $\PP^1$ is a direct sum of line bundles \cite[Ex.~V.2.6]{Har77},
we may choose equations for $X$ inside some $\FF_1(a_0,a_1,a_2)$
with $0 \leq a_0 \leq a_1 \leq a_2$ as a smooth hypersurface of bidegree $(e,2)$, for some $e \geq 0$.
We are thus in the setting of Theorem \ref{thm:conic_bundle_eqn}.
On noting that $8 - K_X^2 = 2(a_0 + a_1 + a_2) + 3e$ \cite[Prop.~2.5]{FLS16}, we therefore have
$$\frac{8 - K_X^2}{3} = \frac{2(a_0 + a_1 + a_2)}{3} + e \geq a_0 +  a_1 + e.$$
Thus $\alpha > (8 - K_X^2)/3$ implies that $\alpha > a_0 + a_1 + e$, so 
applying Theorem \ref{thm:conic_bundle_eqn}, with $E$ the union of the singular fibres and the hypersurface $x_2 = 0$,
gives the result. \qed

\subsection{Proof of Theorem \ref{thm:dP}}
As above we embed $X$ inside $\FF_1(a_0,a_1,a_2)$,
for some $0 \leq a_0 \leq a_1 \leq a_2$, as a smooth hypersurface of bidegree $(e,2)$.
For $d \geq 6$ we may apply Theorem \ref{thm:conic_bundle}. For $5 \leq d \leq 3$ and
$ d < 3$ it is shown in \cite[Thm.~5.6]{FLS16} that the invariants may be chosen to satisfy
$a_0 + a_1 +e = 1$ and $a_0 + a_1 +e = 2$, respectively. Thus we may apply Theorem \ref{thm:conic_bundle_eqn} in these cases.

It remains to show that $S: x_2 = 0$ is not an accumulating subvariety. 
Let $C$ be an irreducible component of $S$. As $S$ is a multisection of $\pi$,
we see that $C$ is also a multisection of $\pi$. It follows that
$C\cdot F\geq 1$.
Moreover, as $-K_X$ is ample,
we have $C \cdot (-K_X) \geq 1$. We deduce that
$$C \cdot (-K_X + \alpha F) \geq 1 + \alpha > 2.$$
Standard results for counting rational points on curves (see e.g.~\cite[\S 9.7]{Ser97})  show that
$C$ contains 
$O_C(B^{2/(1 + \alpha)})$
rational points of bounded height, hence $C$ does not effect the main term of the 
asymptotic formula, as claimed. \qed

%As explained on mathoverflow, $\widetilde{X}$ naturally lives inside $\FF_n(0,0,1)$ as a hypersurface
%of bidegree $(1,2)$. Namely, it has the form
%\begin{equation} \label{eqn:conic_bundle}
%	\sum_{0 \leq i,j \leq 2} f_{i,j}(y_0,y_1,y_2)x_ix_j= 0
%\end{equation}
%where the degrees of the $f_{i,j}$ are given by the following matrix
%\begin{equation} \label{eqn:degree_matrix}
%\left( \begin{array}{ccc}
%2a_0 + e & a_0 + a_1 + e & a_0 + a_2 + e \\
%a_0 + a_1 + e & 2a_1 + e & a_1 + a_2 + e \\
%a_0 + a_2 + e & a_1 + a_2 + e & 2a_2 + e \end{array} \right).
%\end{equation}

%For a cubic threefold this is
%\begin{equation} \label{eqn:degree_matrix}
%\left( \begin{array}{ccc}
%1 & 1 & 2 \\
%1 & 1 & 2 \\
%2 & 2 & 3 \end{array} \right).
%\end{equation}
%This is the same degrees that one obtains for a cubic surface.

%The height function is given by 
%$$\prod_v \max\{|y_0|_v,|y_1|_v,|y_2|_v\}^{n+1}\max\{|x_0|_v, |x_1|_v, |y_0x_2|_v,|y_1x_2|_v,|y_2x_2|_v\}.$$

%We take $\alpha>2$. Then the height bound $H(y;x) \leq B$ implies that the $y_i$ are very small.
%Applying Lemma \ref{lem:conics_upper_bound}, we find the upper bound
%$$N(C_{y}, B)/B  \ll \max\{|y_i|\}^\varepsilon 
%	\left(\frac{1}{\max\{|y_i|\}^{n + 1/3} |\Delta(y)|^{1/3}} + \frac{1}{H(y)^{n+1}} \right).$$
%We use the same dominated convergence theorem trick. 
%As $\Delta(y) \neq 0$ by assumption, everything convergences fine as  $\alpha > 2$. \qed

\subsection{Proof of Theorem \ref{thm:bi_projective}}
This follows immediately from Theorem \ref{thm:HDP}, on noting that 
the contribution from the rational points in $x_2 = 0$ is negligible. (The coordinate $x_2$
is not special in this case as $a_0 = a_1 = a_2 = 0$). \qed
 
\subsection{Proof of Theorem \ref{thm:cubic}}
Let $X \subset \PP^{n+2}$ be a smooth cubic hypersurface over $k$ with a line $L \subset X$.
Let $P \to \PP^{n+2}$ be the blow-up of $\PP^{n+2}$ in $L$; this is isomorphic to the $\PP^2$-bundle
over $\PP^n$ given by $\FF_n(0,0,1)$ \cite[Prop.~9.11]{EH16}.
Moreover, the strict transform of $X$ inside $P$ is exactly
the blow-up $\widetilde{X}$  of $X$ in $L$. We consider the induced map $\pi: \widetilde{X} \to \PP^{n}$.

We claim that $\widetilde{X}$ has bidegree $(1,2)$
in $\FF_n(0,0,1)$. To verify this, we may assume that $L : z_2=\dots=z_{n+2} = 0$,
where the $z_i$ are coordinates on $\PP^{n+2}$. The blow-up map is then given by
$$\FF_n(0,0,1) \to \PP^{n+2}, \quad (y;x) \mapsto (x_0:x_1:y_0x_2:\dots:y_nx_2).$$
Using this, one easily sees that the strict transform of $X$ has the claimed bidegree.

Using Theorem \ref{thm:HDP}, it suffices to show that the contribution from the 
hypersurface $x_2 = 0$ is negligible in this case. 
(Note that this is the exceptional divisor of the blow-up). 
Let $H^*$ be the height from Lemma \ref{lem:standard_height}. Then 
\begin{align*}
	\#\{ (y;x) \in \widetilde{X}(k): H^*(y;x) \leq B, x_2 = 0 \} 
	\leq 2 \#\{ y \in \PP^n(k) : H(y)^{n + 1 + \alpha - 2} \leq B\} = o(B), 
\end{align*}
since $\alpha > 2$ by assumption. 
The result therefore follows from Theorem \ref{thm:HDP}. \qed

\section{Compatibility with conjectures} \label{sec:conjectures}

We now explain the compatibility of our results with the Batyrev--Manin conjecture \cite{BM90}
and Batyrev--Tschinkel's conjecture \cite{BT_Tamagawa} for the leading constant.

\subsection{Batyrev--Manin}

Let $X$ be a smooth projective rationally connected variety over a field $k$ of characteristic $0$
and $D$ a big $\QQ$-divisor on $X$. Let $\Lambda_{\mathrm{eff}}(X)$ be the pseudo-effective cone of $X$, i.e.~the closure of the cone of effective divisors on $X$.
We recall that the constants $a(D)$ and $b(D)$ from \eqref{conj:BM} are conjecturally given by
$$a(D) = \inf\{a \in \RR : aD + K_X \in \Lambda_{\mathrm{eff}}(X)\}$$
and $b(D)$ is the codimension of the minimal face of $\Lambda_{\mathrm{eff}}(X)$ which contains
the adjoint divisor $a(D)D + K_X$.

\begin{lemma} \label{lem:a_b}
	Let $\pi: X \to \PP^n$ be a proper morphism over a field $k$ of characteristic $0$ whose generic fibre is isomorphic to a plane conic,
	with $X$ non-singular.
	Let $F$ be the pull back of the hyperplane class
	and let $\alpha \in \QQ_{> 0}$ be such that $D= -K_X + \alpha F$ is big. Then
	$$a(D) = 1, \quad b(D) = 1.$$
\end{lemma}
      
\begin{proof}	
	As $D + K_X = \alpha F \in \Lambda_{\mathrm{eff}}(X)$, we clearly have $a(D) \leq 1$.
	So let  $\varepsilon > 0$ and assume that the $\RR$-divisor
	$$ P:=(1 - \varepsilon)D + K_X$$%= \alpha(1 - \varepsilon)F + \varepsilon K_X $$
	is pseudo-effective. Then we have
 	$$\alpha F = D + K_X = \varepsilon D + P.$$
	As the sum of a big divisor and a pseudo-effective divisor is big,
	this implies that $\alpha F$ is big (this follows from the fact
	that the big cone is the interior of the pseudo-effective cone).
	However $\alpha F$ is clearly not big, as the map $\pi$ is not birational;
	contradiction.
%	Alternative proof: let $\varepsilon > 0$
%	and assume that the divisor
%	$$ A:=(1 - \varepsilon)D + K_X = t(1 - \varepsilon)F + \varepsilon K_X $$
%	is effective. Let $f$ be a smooth fibre over some rational point.
%	Note that $F \cdot f = 0$. Moreover, the relative anticanonical bundle is given by 
%	$\omega_\pi^{-1} = \omega_X^{-1} \otimes \OO_X((-n-1)F)$. As $\omega_\pi^{-1}$ induces the anticanonical bundle
%	on $f$, we find that $f \cdot (-K_X) = 2$. Thus  $A \cdot f = -2\varepsilon < 0$.
%	However, as $f$ is a movable curve and $A$ is effective, we have $A \cdot f \geq 0$,
%	which is a contradiction.
	
	The adjoint divisor is thus $a(D)D + K_X = \alpha F$.	
	To calculate $b(D)$ we use \cite[Prop.~18]{HTT15}. As $X$
	is rationally connected \cite[Cor.~1.3]{GHS03} we have $\Pic X = \NS X$. Thus it follows from \cite[Prop.~18]{HTT15} that
	\begin{equation} \label{eqn:b(D)}
		b(D) = \rank \Pic X - \rank \Pic_\pi X,
	\end{equation}
	where $\Pic_\pi X \subset \Pic X$ is the sublattice of $\pi$-vertical divisors, i.e.~classes of divisors
	$E \subset X$ such that $\pi(E) \neq \PP^n$. We claim that there is an exact sequence 
	\begin{equation} \label{seq:exact}
		0 \to \Pic_\pi X \to \Pic X \to \Pic X_\eta \to 0,
	\end{equation}
	where $X_\eta$ denotes the generic fibre of $\pi$.
	Exactness on the left is clear, whereas exactness on the right follows from simply taking the closure
	in $X$ of any divisor on $X_\eta$. For exactness in the middle, let $E$ be a divisor whose
	restriction to the generic fibre is principal, i.e.~there is a rational function $f$
	on $X_\eta$ such that $E|_{X_\eta} = \div_{X_\eta} f$. 
	But $f$ is equally well a rational function on $X$,
	hence we have $E - \div_{X} f  \in \Pic_\pi X$. 
	It follows that $[E] \in \Pic_\pi X$, which shows that
	\eqref{seq:exact} is indeed exact.
	
	As $X_\eta$ is just a conic, we have $\rank \Pic X_\eta = 1$. Therefore \eqref{eqn:b(D)}
	and the exactness of \eqref{seq:exact} imply that $b(D) = 1$, as required.	
%	$\mathbf{\dim X = 2}$: 
%	Recall that 
%	\begin{equation} \label{eqn:rho}
%		\rho(X) =  2 + \#\{ \text{closed points }x \in \PP^1 : \pi^{-1}(x) \text{ is singular and  split} \}
%	\end{equation}
%	(see \cite[Lem.~2.1]{FLS16}).
%	However the divisor $\alpha F$ lies in the face of $\Lambda_{\mathrm{Eff}}(X)$ containing the divisors
%	$$\{F\} \cup \{ D : D \subset \pi^{-1}(x) \text{ is singular and  split}, \, x \in \PP^1\}.$$
%%	By \eqref{eqn:rho} this face has codimension $1$, thus again $b(D) = 1$.	
%	Alternative idea: Use the fact that $b(D)$ is a birational invariant. Hence without loss of generality
%	we may reduce to the case of standard conic bundles, where the proof is clear.	
\end{proof}

Lemma \ref{lem:a_b} shows that the asymptotic formulae we obtain in this paper agree with the conjecture \eqref{conj:BM}.

\subsection{Batyrev--Tschinkel}
The leading constant in Theorem \ref{thm:main} is equal to the sum of the Peyre constants of those smooth fibres of $\pi: X \to \PP^n$ which meet $U$.

This is in agreement
with the conjectural constant proposed by Batyrev and Tschinkel in \cite[\S 3.5]{BT_Tamagawa}. Namely,
we are in the situation of Case $1$ of \cite[\S 3.5]{BT_Tamagawa}, and, as explained there, the leading
constant should be given as the sum of the leading constants of each of the smooth fibres (in the terminology of 
\cite{BT_Tamagawa}, our variety $X$ is not ``strongly $\mathcal{L}$-saturated'' and not ``$\mathcal{L}$-primitive'', but the map $\pi$ is an ``$\mathcal{L}$-primitive
fibration'').

\smallskip
However, in \cite[Conj.~3.5.1]{BT_Tamagawa} is stated a related conjecture, which turns out to \emph{fail} to hold
in our case. This conjecture is fairly general; we make it explicit in the case of conic bundles considered in this paper.

\begin{conjecture}[Batyrev-Tschinkel] \label{conj:BT}
	Let $\pi: X \to \PP^n$ be a conic bundle over a number field $k$ and 
	let $F$ be the pull back of the hyperplane class.
	Let $\alpha \in \QQ_{> 0}$ be such that $D= -K_X + \alpha F$ is big and choose an adelic
	metric on $\OO_X(D)$. Let $H$ be the usual $\OO(1)$-height 	on $\PP^n$. 
	Then there exists $c_2 > c_1 > 0$ and $U \subset \PP^n$ dense open such that 
	for all $y \in U(k)$ we have
	$$\frac{c_1}{H(y)^{n + 1 + \alpha}} \leq \tau(X_y) \leq \frac{c_2}{H(y)^{n + 1 + \alpha}}.$$
	Here $\tau(X_y)$ denote the Tamagawa measure of the fibre $X_y$ with respect to the
	adelic metric induced by $\OO_X(D)$.
\end{conjecture}

We illustrate the failure of \emph{both} inequalities in this conjecture for the 
hypersurface $X$ over $\QQ$ defined by
\begin{equation} \label{eqn:X}
	Q(s,t,x_0,x_1,x_2) = x_0^2 + x_1^2 - stx_2^2 \quad \subset \FF_1(0,0,1),
\end{equation}
with respect to the $\OO_X(D)$-height $H^*$ from Lemma \ref{lem:standard_height}.

The method we present can be generalised without too much difficultly to the more general set up of Conjecture \ref{conj:BT}.
Counter-examples to the upper bound in a different setting have been found by Derenthal and Gagliardi \cite{DG16}, however counter-examples
to the lower bound appear here for the first time.

We consider the fibres $X_{t}$ above points $(1:t)\in\PP^1(\QQ)$, which are isomorphic to the plane conics $C_t$ defined by $x_0^2+x_1^2=tx_2^2$. First note that if $t$ is prime and $t \equiv 3 \bmod 4$ then the lower bound of the conjecture clearly fails: the 
corresponding conic has no rational point, so the Tamagawa measure of the fibre is $0$ but the lower bound in Conjecture \ref{conj:BT} is positive.
As for the upper bound, we have the following.

\begin{lemma}
	Let $t$ be a positive squarefree integer whose prime divisors are all $1 \bmod 4$.
	Then
	$$\tau(X_t) = \frac{\pi}{t^{2+\alpha}} \prod_{p \mid t} 2\left(1 - \frac{1}{p}\right) \prod_{p \nmid 2t} \left(1 - \frac{1}{p^2}\right).$$
In particular, for such $t$ we have 
	$$\tau(X_t) \geq \frac{\pi}{\zeta(2)}\frac{(4/3)^{\omega(t)}}{t^{2+\alpha}},$$
	where $\omega(t)$ denotes the number of prime factors of $t$.
\end{lemma}
\begin{proof}
  The conic $X_t$ has a rational point in this case, by Fermat's theorem. The height on $X$ pulls back along the isomorphism $C_t\cong X_t$ to the height $H=\prod_{v}H_v$, where
  \begin{align*}
    H_\infty(\vx) &= t^{\alpha+1}\max\{|x_0|,|x_1|,|tx_2|\}, \text{ and}\\
    H_p(\vx) &= \max\{|x_0|,|x_1|,|x_2|\}\text{ for prime }p.
  \end{align*}
The Tamagawa number has the form
\begin{equation*}
  \tau(X_t) = \tau(C_t) = \sigma_\infty\prod_p\sigma_p,
\end{equation*}
with $\sigma_\infty,\sigma_p$ the local densities. We compute $\sigma_\infty$ using \cite[Lem.~5.4.4]{peyre}. We may restrict ourselves to the open subsets $U_\pm:=\{(x_0:x_1:1)\mid \pm x_1>0\}\subseteq C_t(\RR)$, since the complement of their union is finite. With the obvious charts $\rho_\pm : U_\pm \to (-\sqrt{t},\sqrt{t})$, $(x_0:x_1:1)\mapsto x_0$, we get $\sigma_\infty = \sigma_+ + \sigma_-$, where
\begin{equation*}
  \sigma_{\pm}=\int_{-\sqrt{t}}^{\sqrt{t}}\frac{\mathrm d x}{H_\infty(\rho_\pm^{-1}(x))\abs{(\partial Q / \partial x_1)(1,t,\rho_\pm^{-1}(x))}} =\int_{-\sqrt{t}}^{\sqrt{t}}\frac{\mathrm d x}{t^{2+\alpha}2\sqrt{t-x^2}} = \frac{\pi}{2t^{2+\alpha}}.
\end{equation*}
Thus, $\sigma_\infty=\pi/t^{2+\alpha}$. For the $p$-adic densities, we let
\begin{equation*}
  N(p^n) := \card\{\vx\bmod p^n\mid \vx\not\equiv\vz\bmod p,\ x_0^2+x_1^2\equiv tx_2^2\bmod p^n\}.
\end{equation*}
Then it is well-known that (cf.~\cite[Cor.~3.5]{PT01})
\begin{equation*}
  \sigma_p = \lim_{n\to\infty}\frac{N(p^n)}{p^{2n}}.
\end{equation*}
  Let $p$ be an odd prime. If $p\nmid t$ then it is well-known that $\sigma_p = (1-p^{-2})$.
  If $p \mid t$, the proof of \cite[Prop.~3.6]{FLS16} shows that
  \begin{equation*}
    \sigma_p= 2(1-p^{-1}) \geq (4/3)(1 - p^{-2}).
  \end{equation*}
  For $p=2$, an application of Hensel's lemma shows that $N(2^n) = 2^{2n-6}N(8)$ for all $n\geq 3$, so $\sigma_2=N(8)/64$. One can verify by direct calculations that $N(8)=64$ in both possible cases $t\equiv 1,5\bmod 8$. Thus, $\sigma_2=1$.
\end{proof}

This shows the failure of the upper bound in Conjecture \ref{conj:BT}.
However  the problem just lies with the non-archimedean densities, and in our case we have 
$$\tau(X_y) \ll_\varepsilon \frac{H(y)^{\varepsilon}}{H(y)^{2 + \alpha}} $$
for all $\varepsilon > 0$. So the upper bound in Conjecture \ref{conj:BT} does not fail ``too badly''. 
It is for this reason that the sum of Peyre constants in Theorem  \ref{thm:main} is  still convergent.

\bibliographystyle{amsalpha}
\bibliography{fibration}
\end{document}